\documentclass{amsart}
\usepackage{color}
\usepackage{amsmath}    
\usepackage{graphicx}   
\usepackage{verbatim}   
\usepackage{amssymb}
\usepackage{amsthm}
\usepackage[textwidth=468pt, textheight=622pt]{geometry}
\usepackage{mathrsfs}
\usepackage{bbm}
\usepackage{exscale}
\usepackage{enumerate}
\usepackage{cite}
\usepackage{multicol}
\usepackage{adjustbox}
\usepackage{epstopdf}
\usepackage[normalem]{ulem}
\usepackage{xcolor}

\theoremstyle{plain}
\newtheorem{theorem}{Theorem}[section]
\newtheorem*{theorem*}{Theorem}
\newtheorem{lemma}[theorem]{Lemma}
\newtheorem{corollary}[theorem]{Corollary}
\newtheorem{proposition}[theorem]{Proposition}

\newtheorem{problem}[theorem]{Problem}

\theoremstyle{remark}

\theoremstyle{definition}

\newtheorem{definition}[theorem]{Definition}
\newtheorem*{Acknowledgements*}{Acknowledgements}

\newcommand{\finsum}[3]{
\underset{#1=#2}{\overset{#3}\sum}}

\newcommand{\dint}{
\displaystyle\int}

\newcommand{\dsum}{
\displaystyle\sum}

\newcommand{\zsumnzero}[1]{
\underset{#1\neq 0}\dsum}

\newcommand{\R}{
\mathbb{R}}
\newcommand{\C}{
\mathbb{C}}
\newcommand{\N}{
\mathbb{N}}
\newcommand{\Z}{
\mathbb{Z}}

\newcommand{\supp}{
\textnormal{supp}}

\newcommand{\schwartz}{
\mathscr{S}}
\newcommand{\bracket}[1]{
\left\langle#1\right\rangle}

\newcommand{\phica}{
\widehat{\phi_{\alpha,c}}}

 \newcommand{\eps}{
 \varepsilon}

 \newcommand{\floor}[1]{
 \lfloor#1\rfloor}

 \title{Regular Families of Kernels for Nonlinear Approximation}
 
\author{Keaton Hamm}
\address{Department of Mathematics, University of Arizona, Tucson, AZ, 85721}
\email{hamm@math.arizona.edu} 

\author{Jeff Ledford}
\address{Department of Mathematics and Computer Science, Longwood University, Farmville, VA, 23901}
\email{ledfordjp@longwood.edu}



\begin{document}

\begin{abstract}

This article studies sufficient conditions on families of approximating kernels which provide $N$--term approximation errors from an associated nonlinear approximation space which match the best known orders of $N$--term wavelet expansion.  These conditions provide a framework which encompasses some notable approximation kernels including splines, cardinal functions, and many radial basis functions such as the Gaussians and general multiquadrics.  Examples of such kernels are given to justify the criteria.  Additionally, the techniques involved allow for some new results on $N$--term Greedy interpolation of Sobolev functions via radial basis functions.
\end{abstract}
 
\maketitle
\allowdisplaybreaks

\section{Introduction}

Among the many notions of smoothness spaces of functions on $\R^d$ are spaces of fractional smoothness which may be defined by, or intimately related to, a wavelet expansion.  The primary examples to be considered here are the {\em Besov} and {\em Triebel--Lizorkin} spaces.  For an overview of the basic facts about these spaces, the reader is invited to consult \cite{Kalton} and the many references therein.  DeVore, Jawerth, and Popov \cite{DJP} studied approximation orders for best $N$--term approximation of functions from these spaces with respect to a given wavelet system; in particular, it was shown that if $f$ is in the Triebel--Lizorkin space $F_{\tau,q}^s$ (where $s\in\R_+$ is the smoothness) then the error of the best $N$--term approximation in $L_p$ (for $p = (1/\tau-s/d)^{-1}$) via many wavelet systems is $O(N^{-s/d})$.

This article studies $N$--term approximation of functions in these smoothness spaces from nonlinear spaces associated with different kernels, which are typically related to a spline system or a {\em radial basis function}, i.e. one for which $\phi(x)=\psi(|x|)$ for a univariate function $\psi$.

More specifically, suppose we have a family of continuous kernels depending upon some parameter, $(\phi_\alpha)_{\alpha\in A}$ for some unbounded $A\subset(0,\infty)$.  Considering the following nonlinear approximation space:
\begin{equation}\label{EQPHIN}\Phi_N:=\left\{\finsum{j}{1}{N}a_j\phi_{\alpha_j}(\cdot-x_j)\;:\;(a_j)\subset\C, (\alpha_j)\subset A, (x_j)\subset\R^d\right\},\end{equation}
we seek to answer the following problem.
\begin{problem}\label{PROB}
Find sufficient conditions on the family $(\phi_\alpha)$ such that for $f\in F_{\tau,q}^s$, 
there exists an $S_{f,N}\in\Phi_N$ such that
\[\|f-S_{f,N}\|_{L_p} \leq C_{f} N^{-s/d}.\]
\end{problem}

Similar problems have appeared in various contexts throughout the literature.  
In the shift invariant setting, de Boor, DeVore, and Ron investigated similar problems, \cite{DDR1, DDR2}.  The methodology used in those papers exploits features of the periodization of $\vert\widehat{\phi}\vert^2$ to provide approximation results analogous the one above.  In \cite{KP}, Kyriazis and Petrushev develop sufficient conditions for unconditional bases in Triebel--Lizorkin spaces involving moment conditions from which an answer to their version of the above problem arose.  Interestingly, although their techniques are different, there is overlap in the examples that they produce.  In \cite{DR}, DeVore and Ron study approximation from both linear and nonlinear spaces using kernels that arise as the solution of certain elliptic differential operators. This technique encompasses some growing kernels, but does not allow the use of some notable examples, such as multiquadrics.  Hangelbroek and Ron \cite{HR} analyzed best $N$--term approximations via Gaussians.  The analyticity and rapid decay of the Gaussian kernel allows for very fast approximation in theory; however, their analysis does not directly extend to allow kernels of finite smoothness, and certainly not growing kernels.  Additionally, Buhmann and Dai \cite{BD} carried out a related analysis for quasi-interpolation schemes based around polynomial reproduction.  Our approximation scheme (detailed in Section 3) is similar to theirs in that it is essentially one of quasi-interpolation, but the kernel they use is formed indirectly from the initial kernel to guarantee polynomial reproduction, and so is of a different form.
Our main results give sufficient conditions on finite-smoothness kernels (both decaying and growing ones) which allow for the desired approximation orders using properties of the kernel.
Some notable implications of the current analysis are that other growing kernels such as the {\em multiquadrics} of all orders may be used, and also our methods allow for the use of {\em cardinal functions}, which gives rise to some new nonlinear methods involving interpolation.  The use of cardinal functions in turn allows for error estimates for Sobolev spaces, which cannot be achieved directly from the Besov and Triebel--Lizorkin space estimates. 

The primary concern of this work is to demonstrate that the approximation orders of the nonlinear spaces $\Phi_N$ associated with a large variety of kernels match the known rates for nonlinear wavelet systems, which are determined by the smoothness of the target functions.  However, in the course of the proof, we exhibit a concrete approximant to a given function $f$ which attains the error bound of $O(N^{-s/d})$.  As an intermediate step in the procedure, we use the fact that such $f\in F_{\tau,q}^s$ admits a wavelet expansion of a certain type, though the choice of the wavelet is not overly important, and we stress that the use of wavelets here is merely an intermediate step in the analysis.  

Specifically, we employ a two-stage approximation process as follows:

{\bf Step 1:}  Form a linear approximation space which provides extremely good error bounds for approximating mother wavelets which are bandlimited Schwartz functions.  Then truncate the approximant.  If $\psi$ is the wavelet, let $T_N\psi$ be its truncated approximant from the linear space, which is also an $N$--term approximant in $\Phi_N$.

{\bf Step 2:}  Given a wavelet expansion $f=\sum_I f_I\psi_I$, where $I$ are dyadic cubes, and $\psi$ is as in Step 1, approximate $S_{f,N} = \sum_I f_IT_{N_I}\psi_I$, for some cost distribution $(N_I)$ with $\sum_I N_I \leq N$.

For a particular cost distribution (i.e. choice of $(N_I)$ based on the wavelet coefficients $(f_I)$), this $S_{f,N}$ obtains the desired error bounds for the given $f$.  In general terms, we now state the main theorems contained below (Theorems \ref{THMLpMain} and \ref{THMGrowingMain}):

\begin{theorem*}
Suppose that $(\phi_\alpha)_{\alpha\in A}$ is a family of kernels satisfying conditions (A1)--(A6) or (B1)--(B4) below.  Then for the appropriate choices of parameters $s,\tau,q$, and $p$, every $f\in F_{\tau,q}^s$  has an $N$--term approximant $S_{f,N}\in\Phi_N$ such that 
$$\|f-S_{f,N}\|_{L_p}\leq CN^{-s/d}|f|_{F_{\tau,q}^s}.$$
\end{theorem*}

As a final note, let us point out that the approximation spaces of the form $\Phi_N$ considered here are both nonlinear as mentioned above, but also nonuniform in the sense that no structure is required \textit{a priori} on the shifts $(x_j)$ -- although the particular method illustrated in the sequel will use regular shifts at different levels.  This is in contrast to the shift-invariant space literature, in which the approximation spaces are of the form $\overline{\text{span}}^{L_p}\{\phi(\cdot-j):j\in\Z^d\}$, though the resulting approximation bounds are similar.  For approximation and structural results of these spaces, consult \cite{Bownik,DDR1,DDR2,Jia,Johnson,Johnson2}.  Typical results therein apply to approximating (or interpolating) Sobolev classes, and provide approximation rates of order $h^k$, where $k$ is the specified smoothness of the Sobolev space.  Approximation spaces which are nonuniform but linear, i.e. of the form $\overline{\text{span}}^{L_p}\{\phi(\cdot-x_j):j\in\Z^d\}$, have been studied for quite some time as well.  In recent works, these have been called quasi shift-invariant spaces, e.g. \cite{Atreas,DV,GS,HLQuasi}.  For a more extensive history, Wendland's text on scattered-data interpolation is an excellent reference \cite{Wendland}.  For approximation orders for interpolation of Sobolev functions similar to the theorem above, consult \cite{Hamm,Ledford}. 

The rest of the paper is laid out as follows: Section \ref{SECPrelimin} begins with some basic definitions, followed by the conditions on decaying kernels for rapid approximation of bandlimited Schwartz functions in Section \ref{SECCriteria}.  Approximation orders using nonlinear approximation spaces associated with these kernels are given in Sections \ref{SECLp} and \ref{SECLinfty}, and examples of kernels satisfying the regularity conditions are given in Section \ref{SECExamples}.  Section \ref{SECGrowing} provides sufficient conditions on growing kernels to allow for similar approximation rates, and examples of such families of kernels are given in Section \ref{SECGrowingExamples}.  In Section \ref{SECInterpolation}, we move the discussion to the related topic of cardinal interpolation, which allows us to give approximation orders for broader classes of smooth functions, namely Sobolev spaces. This analysis also develops some Greedy interpolation schemes related to Gaussians and multiquadrics.  To give a better idea of the limitations, a special case of the cost distribution for the approximation scheme is discussed in more detail in Section \ref{SECCost}.  

\section{Preliminaries}\label{SECPrelimin}

Let $\schwartz$ be the space of Schwartz functions on $\R^d$, that is the collection of infinitely differentiable functions $\phi$ such that for all multi-indices $\alpha$ and $\beta$, $\underset{x\in\R^d}\sup\left|x^\alpha D^\beta\phi(x)\right|<\infty\;.$
The Fourier transform and inverse Fourier transform of a Schwartz function $\phi$ are given, respectively, by
$$
 \widehat{\phi}(\xi):=\int_{\R^d} \phi(x)e^{-i\bracket{\xi, x}}dx,\quad \xi\in\R^d;\qquad \phi^\vee(x)= \dfrac{1}{(2\pi)^d}\dint_{\R^d}\phi(\xi)e^{i\bracket{x,\xi}}d\xi,\quad x\in\R^d,
$$ where $\bracket{\cdot,\cdot}$ is the usual scalar product on $\R^d$.  Let $\Omega\subset\R^d$ be an open set. Then let $L_p(\Omega)$, $1\leq p\leq\infty$, be the usual space of $p$--integrable functions on $\Omega$ with its usual norm.  If no set is specified, we mean $L_p(\R^d)$.  Additionally, we will use $\N_0$ to denote the set of natural numbers including $0$, and $\R_+$ to denote the positive real numbers, i.e. $(0,\infty)$.

We denote by $F_{\tau,q}^s$ the Triebel--Lizorkin space (defined in Section \ref{SECLp}, where $s\in\R_+$ essentially determines the smoothness of the functions in the space).  For $1\leq p\leq\infty$ and $k\in\N$, let $W_p^k(\Omega)$ be the Sobolev space of functions in $L_p(\Omega)$ whose weak derivatives of order up to $k$ are also in $L_p(\Omega)$.  The norm  and seminorm on the Sobolev space may be defined, respectively, via $$\|f\|_{W_p^k(\Omega)}:=\|f\|_{L_p(\Omega)}+|f|_{W_p^k(\Omega)},\quad \textnormal{and}\quad |f|_{W_p^k(\Omega)} = \max_{|\beta|=k}\|D^\beta f\|_{L_p(\Omega)}.$$

Let $PW_S$ be the Paley--Wiener space of bandlimited $L_2$ functions whose Fourier transform is supported on $S$.  If $B(0,R)$ is the Euclidean ball of radius $R$ centered about the origin, define the space $\schwartz_B:=\schwartz\cap PW_{B(0,R)}$; in the sequel, $R$ will be determined from the context.  For a set $S$, let $\chi_S$ be the function which takes value 1 on $S$ and 0 elsewhere, and let $|S|$ denote the Lebesgue measure of the set $S$.  We use $C$ to denote constants, and where appropriate add subscripts to emphasize the parameters they may depend on.  

\section{Regularity Criteria and Approximation in $\schwartz_B$}\label{SECCriteria}

We begin with Step 1 described above, which is to give criteria on kernels which allows for rapid approximation of bandlimited Schwartz functions from a linear space involving the kernels.  For now, we resolve our problem for decaying kernels, and turn to growing ones in later sections.  Let $A$ be an infinite subset of $(0,\infty)$, and $(\phi_\alpha)_{\alpha\in A}$ be a family of continuous kernels depending on the parameter $\alpha$, and let $\Phi_N$ be as defined in \eqref{EQPHIN}.

Letting $R>0$ be fixed, but arbitrary, and $B:=B(0,R)$, we introduce the following regularity conditions and assume that they hold from here on unless otherwise noted.
\begin{enumerate}
    \item[(A1)] $\Phi_N$ is closed under translation and dilation.
    \item[(A2)] Each $\phi_\alpha$ is continuous, and $\underset{\alpha\in A}\sup\|\phi_\alpha\|_{L_\infty}\leq C$.
    \item[(A3)] Let $A_{\alpha,h,j}:=\dfrac{\widehat{\phi_\alpha}(\cdot+{2\pi j}/{h})}{\widehat{\phi_\alpha}}.$ Suppose that $\zsumnzero{j}\|A_{\alpha,h,j}\|_{L_{\infty}(B)}\leq g_\alpha(h)$ for some function $g_\alpha(h)$, and that for every $k\in\N_0$, there exists an $\alpha_k\in A$ such that for every $\alpha\geq\alpha_k$, $g_\alpha(h)\leq Ch^k$ for some absolute constant $C$.
    \item[(A4)] For every $k\in\N$, there exists an $\alpha_k'\in A$ such that for every $\alpha\geq\alpha_k'$, $D^\gamma(1/\widehat{\phi_\alpha})\in L_2(B)$ for every $|\gamma|\leq k$.
    \item[(A5)] For every $k\in\N$, there exists an $\alpha_k''\in A$ such that for every $\alpha\geq\alpha_k''$, $|\phi_\alpha(x)| = O(|x|^{-2k}),$ $|x|\to\infty$.
    \item[(A6)] For every $\alpha\in A$, $|\phi_\alpha(x)|+|\widehat{\phi_\alpha}(x)|\leq C(1+|x|)^{-d-\eps}$ for some $C,\eps>0$.
\end{enumerate}

The analysis in \cite{HR} is aided by the fact that the Gaussian has rapid decay -- a requirement we drop in order to consider a family of general inverse multiquadrics which serves as the typical example of a collection of kernels which satisfies (A1)--(A6).  These conditions allow us to provide norm and pointwise estimates for functions in $\schwartz_B$.  The requirement (A1) is to ensure the ease of approximation of the wavelet basis, while (A6) allows one to use the Poisson summation and Fourier inversion formulas (of course (A6) may be replaced by any other condition which allows these to hold).  The utility of the remaining criteria will reveal itself more clearly in the sequel, but essentially (A2)--(A5) guarantee that after some value of the parameter $\alpha$, the kernels are sufficiently localized, smooth, and decaying.  Since the Triebel--Lizorkin spaces are defined by wavelet systems, these criteria guarantee that for any resolution level of the wavelet, the parameter may be chosen so that the kernel is sufficiently localized.  There are other feasible ways to determine such criteria; for instance, various sets of sufficient conditions can be found in \cite{KP, HR, DR, BD}, but our aim here is to give readily checked conditions based on the kernels themselves.  Presently, we define an approximation suited to our purposes.  

For $f\in \schwartz_B$ and $\phi\in (\phi_\alpha)$, we first define $f_\phi$ to be the function such that $\widehat{f_\phi}=\widehat f/\widehat\phi$.  The main obstruction is that $f_\phi$ need not be smooth.  Nevertheless,  we define
\[
T_h^{\sharp}f(x):=\sum_{j\in\mathbb{Z}^d}f_\phi(hj)\phi(x-hj).
\]


\begin{proposition}\label{PROP1}
If $\alpha\geq\alpha_0$ and $h<\pi/R$, then for every $f\in \schwartz_B$,
$$\|f-h^dT_h^\sharp f\|_{L_\infty}\leq C\|\widehat{f}\|_{L_1}g_\alpha(h),$$
where $C>0$ is independent of $f$ and $h$.
\end{proposition}  

\begin{proof}
By the inversion formula and the Poisson summation formula (both are justified by (A6)),
\begin{align*}
h^dT_h^\sharp f(x)& = \dfrac{h^d}{(2\pi)^d}\dint_{\R^d}\hat{f}(\xi) \left[ \dfrac{1}{\widehat{\phi}(\xi)} \sum_{j\in\mathbb{Z}^d}\phi(x-hj)e^{i\langle \xi, hj\rangle}  \right]   d\xi \\ 
& = \dfrac{1}{(2\pi)^{d}}\int_{\mathbb{R}^d} \hat{f}(\xi)\left[ \dfrac{e^{i\langle \xi,x \rangle }}{\widehat{\phi}(\xi)}\sum_{j\in\mathbb{Z}^d}\widehat{\phi}(\xi+2\pi j /h)e^{i\langle x, 2\pi j /h \rangle}     \right]d\xi.
\end{align*}
The $j=0$ term is nothing but $f(x)$, hence
\[
f(x)-h^d T_h^{\sharp}f(x)=\frac{1}{(2\pi)^d}\int_{\mathbb{R}^d}\hat{f}(\xi)\left[ \dfrac{e^{i\langle \xi,x \rangle }}{\widehat{\phi}(\xi)}\sum_{j\neq 0}\widehat{\phi}(\xi+2\pi j /h)e^{i\langle x, 2\pi j /h \rangle}     \right]d\xi.
\]
Recalling that $\supp(\hat{f})\subset B$, we have the straightforward estimate
\[
\| f-h^d T_h^{\sharp}f \|_{L_\infty} \leq \dfrac{1}{(2\pi)^{d}} \left[ \sum_{j\neq 0} \| A_{\alpha,h,j}  \|_{L_\infty(B)}  \right] \| \hat{f} \|_{L_1},
\]
which is at most $C\|\widehat{f}\|_{L_1}g_\alpha(h)$ by (A3).
\end{proof}
Note that by the restriction $\alpha\geq\alpha_k$ above, we mean that for any $\phi\in(\phi_\alpha)_{\alpha\geq\alpha_k}$, the conclusion of the proposition holds.

We move now to truncate our approximant to put it in the space $\Phi_N$.  Define $T_{h}^{\flat}f$ via
\[
T_{h}^{\flat}f(x):= \sum_{j\in\mathbb{Z}^d\cap B_{h}}f_{\phi}(hj)\phi(x-hj),
\]
where $B_{h}$ is the ball of radius $h^{-2}$ centered at the origin.  This is a finite sum of $N\sim h^{-2d}$ terms, hence $T_{h}^{\flat}f\in \Phi_N$.  The next pair of propositions show that this truncated approximant performs quite well both in $L_\infty$ and pointwise.


\begin{proposition}\label{LEMUniformHBound}
Let $k\in\N$ and $f\in \schwartz_B$.  For $h<\pi/R$, there is a constant $C_{f,k}>0$, independent of $h$, so that for all $\alpha\geq\max\{\alpha_k,\alpha_{k+d}'\},$
\[
\| f- h^d T_{h}^{\flat}f  \|_{L_\infty}\leq C_{f,k} h^{k}.
\]
\end{proposition}

\begin{proof}
Note that the norm in question is at most $C\|\widehat{f}\|_{L_1}g_\alpha(h)+h^d\|T_h^\sharp f-T_h^\flat f\|_{L_\infty}$ by the triangle inequality and Proposition \ref{PROP1}, but the first term is at most $C_{f}h^k$ by (A3) and the assumption $\alpha\geq\alpha_k$.  Thus, it remains to show that $\|T_h^\sharp f-T_h^\flat f\|_{L_\infty}\leq Ch^{k-d}$.  By (A2), we have
\begin{equation}\label{EQSharpFlatDiff}
\|T_h^\sharp f-T_h^\flat f\|_{L_\infty}\leq C\dsum_{|j|\geq h^{-2}}|f_\phi(hj)|.
\end{equation}

Now if $\alpha\geq\alpha^{\prime}_{k+d}$, then for all multi-indices $|\gamma|\leq k+d$, $D^\gamma(1/\widehat{\phi})\in L_2(B)$.  Consequently, by standard Fourier transform techniques,
\[|f_\phi(x)|\leq \left(\sum_{|\gamma|\leq k+d}\|D^\gamma (\hat{f}/\hat{\phi} )  \|_{L_1}\right)(1+|x|)^{-k-d},\]
which is, by Leibniz's rule, the Cauchy--Schwarz inequality, the fact that $f\in\schwartz$, and (A4), majorized by
\[\dsum_{|\gamma|\leq k+d}\dsum_{\beta\leq\gamma}\|D^\beta\widehat{f}\|_{L_2}\left\|D^{\gamma-\beta}\left(\frac{1}{\widehat{\phi}}\right)\right\|_{L_2}(1+|x|)^{-k-d}\leq C_{f,k}(1+|x|)^{-k-d}.\]

Thus the series on the right side of \eqref{EQSharpFlatDiff} is at most
\[
C_{f,k}h^{-k-d}\sum_{|j|>h^{-2}}\frac{1}{|j|^{k+d}}\leq C_{f,k}h^{-k-d}\dint_{h^{-2}}^\infty r^{d-1}r^{-k-d}dr = C_{f,k}h^{k-d},
\]
concluding the proof.
\end{proof}


\begin{proposition}\label{prop_ptwise_bnd}
Let $k\in\N$ and $f\in \schwartz_B$.  Then for all $\alpha\geq\max\{\alpha_{2k},\alpha'_{2k+d},\alpha_k''\}$ and sufficiently small $h$, there is a constant $C$, independent of $h$, such that
$$|f(x)-h^dT_h^\flat f(x)|\leq Ch^k(1+|x|)^{-k}.$$
\end{proposition}
\begin{proof}
First notice that if $|x|\leq 2/h$, then $(1+|x|)^{-k}\geq(h/3)^k$, and so the desired inequality arises from Proposition \ref{LEMUniformHBound} (note this requires $\alpha\geq\max\{\alpha_{2k},\alpha'_{2k+d}\}$).

If $|x|>2/h$, $h^k(1+|x|)^{-k}\geq C|x/2|^{-2k}$, and so it suffices to show that $|f(x)-h^dT_h^\flat f(x)|\leq C(1+|x|)^{-2k}$ in this range.  Note that $|f(x)|\leq C(1+|x|)^{-2k}$ since it is in $\schwartz$.  On the other hand,
$$h^d|T_h^\flat f(x)|\leq h^d\dsum_{j\in\Z^d}|f_\phi(hj)|\sup_{|j|\leq h^{-2}}|\phi(x-hj)|.$$
Using the same bound as in the proof of Proposition \ref{LEMUniformHBound}, we see that the series above is bounded as long as $\alpha\geq\alpha'_{d+1}$.  Moreover, in this case, the series is bounded by a constant times $h$ to a positive power, in which case, we may say $h^d\sum_{j\in\Z^d}|f_\phi(hj)|\leq C$ for some constant independent of $h$.  Next, notice that $|x-hj|\geq|x|/2$, and so if $\alpha\geq\alpha''_k$, (A5) implies that $|\phi(x-hj)|=O(|x-hj|^{-2k})=O(|x|^{-2k})$ for sufficiently small $h$.
Therefore, $|h^dT_h^\flat f(x)|\leq C(1+|x|)^{-2k}$ as required.
\end{proof}

Now to make the dependence on $N$ more explicit, we can replace each occurrence of $h$ above with $N^{-1/(2d)}$ and rewrite $T_h^{\flat}f$ as 
\begin{equation}\label{EQTNdef}
T_N f(x):=N^{-1/2}T^{\flat}_{N^{-1/(2d)}}f(x).
\end{equation}

We denote by $N_0$ the smallest such $N$ that satisfies the requirements of Proposition \ref{prop_ptwise_bnd}.  

\section{Approximation in $L_p$, $1\leq p <\infty$ }\label{SECLp}

As suggested in Step 2 above, we use wavelets to bridge the gap between $\schwartz_B$ and $L_p$.  In order to use the results of Section \ref{SECCriteria}, we need a wavelet system whose generators are in $\schwartz_B$ where $B=B(0,R)$ contains the support of the Fourier transform of the mother wavelet.  
Fortunately, such wavelet systems are well known; for example the Meyer wavelet (see  \cite{Daubechies,M,G}) forms one such system.  So that the proofs extend easily to arbitrarily high dimensions (see especially Section \ref{SECGrowing}), we restrict to considering multivariate wavelet systems which are tensor products of univariate ones.  Since we use the same methodology as \cite[Section 3]{HR}, we only list the relevant details adapted to our set up.  Using the notation we find there, $I=c(I)+[0,\ell(I)]^d$ is a cube with corner $c(I)\in\mathbb{R}^d$ and side length $\ell(I)>0$.  Encumbent upon the wavelet structure, we make use of dyadic cubes, i.e.
\[
I\in \mathcal{D}:=\left\{ 2^m(n+[0,1]^d): m\in\mathbb{Z}, n\in\mathbb{Z}^d    \right\}.
\]
Let $\mathcal{D}_j$ be the subset of cubes with common edge-length $2^j$.  We will denote by $\psi_I$ the the natural affine change of variables:
\[
\psi_I(x):=\psi((x-c(I))/\ell(I)).
\]
If $\psi\in \schwartz_B$, then since (A1) is satisfied, we have the following theorem as a direct result of Proposition \ref{prop_ptwise_bnd}.
 
 \begin{theorem}\label{thm_N_bnd}
 Let $k\in\N$, $\psi\in \schwartz_B$, $N>N_0$ and let $I\in\mathcal{D}$.  There exists $\alpha_K \in A$ such that for all $\alpha\geq\alpha_K$, there is a constant $C$, independent of $N$, such that
\[|\psi_I(x)-(T_N \psi)_I(x)|\leq CN^{-k/d}\left(1+\dfrac{|x-c(I)|}{\ell(I)}\right)^{-2k}.\]
 \end{theorem}
 In particular, Theorem \ref{thm_N_bnd} says that to approximate any element of the wavelet system, it suffices to consider the approximant of the mother wavelet under the same affine change of variables.  
 
 We will assume throughout that we have a wavelet system $\Psi$ formed from a mother wavelet $\psi$ as described above.  This allows us to use the results of the previous section because we have (for smooth enough $f$)
 \[
 f=\sum_{I\in\mathcal{D}}f_I\psi_I.
 \]
 To approximate $f$ with $N$ terms, we begin by assigning to each cube $I$ a cost $c_I$, and subsequently a budget
  \begin{equation}\label{EQBudget}
 N_I:= \left\{ \begin{aligned} \lfloor  c_I  \rfloor, & \quad \lfloor c_I \rfloor \geq N_0,\\   0,   & \quad \text{otherwise,}  \end{aligned}     \right.
 \end{equation}
 where $N_0$ is the same as in Theorem \ref{thm_N_bnd} and depends only on the wavelet system.  Requiring that $\sum_IN_I\leq N$, we then simply approximate $f_I\psi_I$ with $N_I$ terms, and set \begin{equation}\label{EQBestApproximant}
 S_{f,N}:=\sum_{I\in\mathcal{D}}f_IT_{N_I}\psi_I
 \end{equation}
 to be our $N$--term approximant in $\Phi_N$ which will exhibit the desired rate of convergence.  The cost distribution depends on several auxiliary quantities, which we define below.
\begin{definition}
Given $s,q>0$, define the maximal function $M_{s,q}f$ via
\[
M_{s,q}f(x):=\left( \sum_{I\in\mathcal{D}}|I|^{-sq/d}|f_I|^q \chi_I(x) \right)^{1/q}.
\]
For a fixed dyadic cube $I$, we define a partial maximal function by
\[
M_{s,q,I}f(x):=\left( \sum_{I\subset I'\in\mathcal{D}}|I'|^{-sq/d}|f_{I'}|^q \chi_{I'}(x) \right)^{1/q}.
\]
Now given $\tau,s,q >0$, we define the \emph{Triebel--Lizorkin space} $F_{\tau,q}^{s}$ via
$$F_{\tau,q}^s:=\{f:|f|_{F_{\tau,q}^s}:= \| M_{s,q}f  \|_{L_\tau(\R^d)}<\infty\},$$ where $|f|_{F_{\tau,q}^{s}}$ is a  quasi-seminorm.
\end{definition}
For the reader unfamiliar with the Triebel--Lizorkin spaces, it should be mentioned that they may be viewed from many different equivalent vantage points.  While the definition above is in terms of mixed sequence and function space norms of a maximal function, they may also be defined via a Littlewood--Paley decomposition, or additionally in terms of a wavelet system; for more details, consult Section 3 of \cite{Kalton}.
Returning to the problem at hand, we are now in position to define the cost distribution which will yield the approximant giving optimal approximation orders.  Note this cost distribution may also be found in \cite{DR,HR}.

\begin{definition}
Let $s>0$ and $1\leq p<\infty$.  Define $\tau$ and $q$ by $1/\tau:=1/p+s/d$ and $1/q:=1+s/d$.  Let $f\in F_{\tau,q}^{s}$, with the wavelet expansion $\sum_{I\in\mathcal{D}}f_I\psi_I$.  The \emph{cost of a dyadic cube} $I\in \mathcal{D}$ is
\[
c_{I}:=|f|^{-\tau}_{F_{\tau,q}^{s}}m_{s,q,I}^{\tau-q}|f_I|^{q}|I|^{q}N,
\]
where $m_{s,q,I}:=\sup_{x\in\mathbb{R}^d}M_{s,q,I}f(x)$. 
\end{definition}
As shown in \cite{HR}, the sum of all costs is bounded by $N$.
Given $(N_I)$, we can use Theorem \ref{thm_N_bnd} to see that we are estimating $f_I\psi_I$ with $N_I$ translates of the kernel $\phi_\alpha$.  Additionally, we can see that for $\alpha\geq\alpha_K$:
\begin{align*}
|\psi_I(x)|&\leq C_{k,d}\min\{1,N^{-k/d}\}\left(1+\dfrac{\text{dist}(x,I)}{\ell(I)}    \right)^{-2k}:=R_I(x);\\
|R_I(x)| &\leq C_{k,d}\min\{1,c_I^{-k/d}    \}\left(1+\dfrac{\text{dist}(x,I)}{\ell(I)}    \right)^{-2k}.
\end{align*}
The proof of Theorem \ref{THMLpMain} relies on the following lemma, whose proof may be found in \cite{HR}.
\begin{lemma}\label{lem8_equiv}
Let $1\leq p <\infty$ and $\alpha\geq\alpha_K$, then
\[
\left\| \sum_{I\in\mathcal{D}} |f_I|R_I  \right\|_{L_p} \leq C_{k,d}\left\| \sum_{I\in\mathcal{D}} \min\left\{1,c_I^{-k/d}\right\}|f_I|\chi_I  \right\|_{L_p}.
\]
\end{lemma}

\begin{theorem}\label{THMLpMain}
Given $s>0$ and $1\leq p <\infty$, there exists a constant $C_{s,p,d}>0$ such that for $f\in F_{\tau,q}^{s}$, with $1/\tau = 1/p + s/d$ and $1/q=1+s/d$, there is a function  $S_{f,N}\in \Phi_N$ so that
\[
\| f-S_{f,N} \|_{L_p}\leq C_{s,p,d}N^{-s/d}|f|_{F_{\tau,q}^{s}}.
\]
In particular, we may take $S_{f,N}$ as in \eqref{EQBestApproximant}.
\end{theorem}
\begin{proof}
The proof is the same as the one given in \cite[Theorem 9]{HR}, \emph{mutatis mudandis}.
\end{proof}

While the proof of the main theorem is essentially the same as in \cite{HR}, it depends on the propositions in Section \ref{SECCriteria} which, on account of conditions (A1)--(A6), allow for the use of kernels that have only finite smoothness.

\section{Approximation in $L_\infty$ }\label{SECLinfty}

It should be noted that the proof of Lemma \ref{lem8_equiv} relies upon the Fefferman--Stein inequality, which is not valid in $L_\infty$.  Thus, for $L_\infty$ convergence phenomena, we restrict our attentions to Besov spaces, which are alternative (slightly smaller) smoothness spaces than the Triebel--Lizorkin spaces which do not rely on maximal functions in their definitions.

Here, we list the relevant results suited to our set up.  It closely matches that of \cite[Section 4]{HR}, thus the proofs are omitted.  
\begin{definition}
For $\tau=d/s\in (0,\infty)$ and $q\in (0,\infty)$, the {\em Besov space} $B^{s}_{\tau,p}$ is defined by the finiteness of the quasi-seminorm
\[|f|_{B^{s}_{\tau,q}}:=\left\| k\mapsto \left( \sum_{I\in\mathcal{D}_k}|f_I|^\tau   \right)^{1/\tau}     \right\|_{\ell_q},\] that is
\[
B^{s}_{\tau,q}:=\left\{ f\in L_\tau : |f|_{B^{s}_{\tau,q}}<\infty   \right\}.
\]
\end{definition}

\begin{lemma}
Let $k>d$ and $\alpha>\alpha_K$.  Suppose that $(a_I)_{I\in D_j}$ is a finitely supported sequence of coefficients; then
\[
\left\| \sum_{I\in\ \mathcal{D}_j}a_I\psi_I - \sum_{I\in\ \mathcal{D}_j}a_I[T_{N_I}\psi_I] \right\|_{L_\infty} \leq C_{k,d} \sup_{I\in \mathcal{D}_j}\left| a_I N_I^{-k/d} \right|.
\]
\end{lemma}

\begin{theorem}
Given $s>0$, there exists a constant $C_{s,d}>0$ such that for $f\in B_{\tau,q}^{s}$, with $1/\tau = s/d$ and $1/q=1+s/d$, there is a function  $S_{f,N}\in \Phi_N$ so that
\[
\| f-S_{f,N} \|_{L_\infty}\leq C_{s,d}N^{-s/d}|f|_{B_{\tau,q}^{s}}.
\]
\end{theorem}

For the cost distribution for the Besov space case that produces the function $S_{f,N}$, see \cite[Section 4.1]{HR}.

\section{Examples}\label{SECExamples}

In this section, we provide numerous examples of kernels which satisfy the conditions prescribed above whose associated $N$--term approximation spaces $\Phi_N$ exhibit optimal rates for nonlinear approximation of smooth functions.

\subsection{Gaussians}\label{SECGaussian}

This example is due to Hangelbroek and Ron and can be found in \cite{HR}.  It served as the inspiration for our analysis.

Here we will let $\phi_{\alpha}$ be the Gaussian:
\[
\phi_\alpha(x):= e^{-|x/\alpha|^2},
\]
whose Fourier transform is given by
\[
\widehat{\phi_\alpha}(\xi)=(\alpha\sqrt{\pi})^d\phi_{2/\alpha}(\xi).
\]
We check conditions (A1)--(A6) below, which are straightforward since the Gaussian is nice.

Condition (A1) is clear from the definition of $\Phi_N$ and the fact that $\phi_\alpha(x/\beta)=\phi_{\alpha\beta}(x)$. The constant in (A2) can be taken to be $\phi_\alpha(0)=1$.  The quantity in (A3) is calculated in \cite{HR} as follows:
\[
\sum_{j\neq 0}\left| \dfrac{\widehat{\phi_\alpha}(\xi+2\pi j/h)}{\widehat{\phi_\alpha}(\xi)}\right| \leq C \widehat{\phi_\alpha}(2(\pi/h-R))\leq C_{\alpha,R}e^{c/h^2},
\] 
where the constant $c>0$ depends on $\alpha$ and $R$.  We may take our parameter $\alpha_k=0$.

Condition (A4) is checked easily by noting that $1/\widehat{\phi_\alpha}=e^\frac{\alpha}{4 |\cdot|^2}$, which is in $L_2^{loc}$ since it is continuous.  As with (A3), we may take $\alpha^\prime_k=0$.  Likewise, conditions (A5) and (A6) are both clearly satisfied (with $\alpha^{\prime\prime}_k=0$) because of the exponential decay of the Gaussian.

\subsection{Inverse Multiquadrics}\label{SECMQI}

We will now consider the general (inverse) multiquadric, 
\[
\phi_{\alpha,c}(x):=(|x|^2+c^2)^{-\alpha},
\]
where $|\cdot|$ denotes the Euclidean distance on $\mathbb{R}^d$; $c>0$ is called the {\em shape parameter}, and we take $\alpha\in A=(d+1/2,\infty)$.
Consider
\[
\Phi_N=\left\{ \sum_{j=1}^Na_j\phi_{\alpha_j,c_j}(\cdot-x_j):(a_j)\subset\mathbb{C},(\alpha_j)\subset A, (c_j)\subset\mathbb{R}^+ ,(x_j)\subset\mathbb{R}^d       \right\}.
\]
Since $\Phi_N$ is manufactured to be closed under translations, we check only dilation.  If $\delta>0$, we have $\phi_{\alpha,c}(\delta \cdot) =  \delta \phi_{\alpha,c/\delta}$; hence $\Phi_N$ is dilation invariant, and (A1) is satisfied.

Condition (A2) is straightforward by noting that $\phi_{\alpha,c}$ is decreasing radially when $|x|\in (0,\infty)$; hence $|\phi_\alpha(x)|\leq c^{-2\alpha}$, which is bounded by $c^{-2d+1}$ for all $\alpha\in A$. 
 Condition (A3) requires the use of the Fourier transform of $\phi_{\alpha,c}$, which is given by (see, for example, \cite[Theorem 8.15]{Wendland}):
\begin{equation}\label{EQgmcft}
 \phica(\xi)=(2\pi)^\frac{d}{2}\dfrac{2^{1-\alpha}}{\Gamma(\alpha)}\left(\dfrac{c}{|\xi|}\right)^{\frac{d}{2}-\alpha}K_{\frac{d}{2}-\alpha}(c|\xi|),\quad \xi\in\R^d\setminus\{0\},
\end{equation}
where
\begin{equation}\label{EQbesseldef}
 K_\nu(r):=\int_0^\infty e^{-r\cosh t}\cosh(\nu t)dt,\quad r>0, \;\nu\in\R.
\end{equation}
The function $K_\nu$ is called the modified Bessel function of the second kind (see \cite[p.376]{AandS}).  We note that $\phi_{\alpha,c}$, and consequently its Fourier transform, are radial functions (i.e. $\phi_{\alpha,c}(x)=\phi_{\alpha,c}(|x|)$). It is also clear from \eqref{EQbesseldef} that $K$ is symmetric in its order; that is, $K_{-\nu}=K_\nu$ for any $\nu\in\R$.  Additionally, the decay of the Bessel function governs the decay of $\widehat{\phi_{\alpha,c}}$, which is exponential away from the origin.

Using the estimates from \cite{HL2}, we may estimate $A_{\alpha,h,j}$ as follows:
\begin{align*}
|A_{\alpha,h,j}(\xi)|\leq & \left( \dfrac{R}{|\xi+2\pi j/h|}   \right)^{d/2-\alpha} e^{cR-c|\xi+2\pi j/h|}\\
\leq & C_{\alpha,R} e^{2cR} e^{-\pi c|j|/h},
\end{align*}
where we have estimated the polynomial term with an exponential.  Thus the series obtained by summing over $j\neq0$ is bounded by $ C_{\alpha,R}e^{2cR}e^{-\pi c /h}$.  Consequently, we may take $\alpha_k > d+1/2$ to satisfy the condition.

Estimates from \cite{HL2} show that $1/\widehat{\phi_{\alpha,c}}$ is only $(2\alpha-d)$--times differentiable near the origin, hence we must have $\alpha^\prime_{k}>k/2+d/2$.  Additionally, each of these derivatives are bounded, hence $1/\widehat{\phi_{\alpha,c}}\in L_2^{loc}$ as desired.  Condition (A5) is obvious by taking $\alpha^{\prime\prime}_k>2k$.  Finally, (A6) follows from the decay of $\phi_{\alpha,c}$ when $\alpha\in A$ and the exponential decay of $\widehat{\phi_{\alpha,c}}$.

As a final remark, our use of the multiquadrics here is not entirely standard in that, typically, a fixed multiquadric order $\alpha$ is used in a given approximation scheme; however, our method detailed above forces the use of multiquadrics with differing parameters $\alpha$ to achieve good approximation for the wavelet expansions in the Triebel--Lizorkin spaces.  On the other hand, if so desired, one could specify a given $\alpha$ to be fixed in the approximation from $\Phi_N$, but only {\em after} knowledge of the cost distribution for the target function $f$.

\subsection{Mat\'{e}rn Kernels}

The Mat\'{e}rn kernels are simply the Fourier transform of the inverse multiquadrics, i.e. $\rho_{\alpha,c}(x):=\widehat{\phi_{\alpha,c}}(x)$, where $\widehat{\phi_{\alpha,c}}$ is given by \eqref{EQgmcft}.  Consequently, $\widehat{\rho_{\alpha,c}}(\xi) = (|\xi|^2+c^2)^{-\alpha}.$  The approximation space associated with these kernels is closed under dilations since $\rho_{\alpha,c}(\delta x) = \delta^{2\alpha-d}\rho_{\alpha,c\delta}(x)$.  It is a simple exercise to verify that condition (A3) holds if $\alpha>d/2$, whilst the other conditions are also readily checked in a fashion similar to the inverse multiquadrics.

\subsection{Cardinal Functions}\label{SECCardinal}

Associated with many kernels are so-called {\em cardinal functions} which exhibit the interpolatory condition on the integer lattice that the cardinal sine function does, i.e. they are functions $L$ such that $L(j)=\delta_{0,j}$, $j\in\Z^d$.  Given a kernel $\phi$, formally define
$$\widehat{L_\phi}(\xi):= \dfrac{\widehat{\phi}(\xi)}{\dsum_{j\in\Z^d}\widehat{\phi}(\xi-2\pi j)},\quad\xi\in\R^d.$$
Then as long as the Fourier inversion formula holds and $\widehat{\phi}$ decays suitably, one may show that $L_\phi$ given by the inverse Fourier transform is a cardinal function.  For sufficient conditions on the kernel $\phi$ for $L_\phi$ to be a cardinal function, see \cite{LedfordCardinal}.  In particular, we assume from here on that $\widehat{\phi}> 0$ on $\R^d$.

Cardinal functions associated with radial basis functions have been studied rather extensively \cite{Baxter,BS,Buhmann,BuhmannBook,HL1,HL2,HMNW,LedfordCardinal,LedfordSpline,RS1,RS2,RS3,RS4,Siva}.  Specifically, the cardinal functions associated with the Gaussian and general multiquadrics are known to have nice decay (exponential in the former case and polynomial based on the exponent $\alpha$ in the latter).

In the examples given below, we also have $L_\phi(x)=\sum_{j\in\Z^d}c_j\phi(x-j)$, where convergence is at least uniform on compact subsets of $\R^d$, but in many cases, the convergence is uniform and in $L_2(\R^d)$.  We note that the impetus for considering such cardinal functions may be found in the ubiquitous literature on spline interpolation (particularly cardinal splines), see for example \cite{Schoenberg,DHR} and references therein.  For some works involving splines used in sampling, see \cite{AldroubiUnser,U1,U2}.

\begin{lemma}\label{LEMCardinalA3}
Let $(\phi_\alpha)_{\alpha\in A}$ satisfy (A3), (denote the quantity there by $A_{\alpha,h,j}^\phi$).  Then the associated cardinal functions $(L_\alpha)_{\alpha\in A}$ satisfy (A3) (denote the associated functions in the condition by $A_{\alpha,h,j}^L$), and moreover
$$\dsum_{j\neq0}\|A_{\alpha,h,j}^{L}\|_{L_\infty(B)}\leq C\dsum_{j\neq 0}\|A_{\alpha,h,j}^\phi\|_{L_\infty(B)}\bigg(1+\dsum_{j\neq0}\|A_{\alpha,1,j}^\phi\|_{L_\infty(B)}\bigg).$$
\end{lemma}

\begin{proof}
Note that 
$$A_{\alpha,h,j}^L(\xi) = \dfrac{\widehat{L_\alpha}(\xi+2\pi j/h)}{\widehat{L_\alpha}(\xi)} = \dfrac{\widehat{\phi_\alpha}(\xi+2\pi j/h)}{\widehat{\phi_\alpha}(\xi)}\dfrac{\dsum_{\ell\in\Z^d}\widehat{\phi_\alpha}(\xi-2\pi\ell)}{\dsum_{k\in\Z^d}\widehat{\phi_\alpha}(\xi+2\pi j/h-2\pi k)}.$$
The series in the denominator on the right is at least $\widehat{\phi_\alpha}(\xi)$ by the assumption of positivity.  Thus
$$A_{\alpha,h,j}^L(\xi)\leq A_{\alpha,h,j}^\phi(\xi)\bigg(1+\sum_{\ell\neq0}A_{\alpha,1,\ell}^\phi(\xi)\bigg),$$
and the required inequality follows from summing over $j\neq0$.
\end{proof}

It follows immediately from this that $g_\alpha^L(h)\leq Cg_\alpha^\phi(h)$, where the functions $g_\alpha$ are the rates assumed in (A3).

For a given set of cardinal functions $(L_\alpha)$, define the $N$--term approximation space via
$$\mathbb{L}_N:=\left\{\finsum{j}{1}{N}a_jL_{\alpha_j}\bigg(\frac{\cdot-x_j}{c_j}\bigg):(a_j)\subset\C, (\alpha_j)\subset A, (x_j)\subset\R^d, (c_j)\subset\R\right\}.$$
Note that these spaces are defined to be closed under translation and dilation, so (A1) is automatically satisfied; this is done because the cardinal functions themselves are not generally preserved under these operations.  However, defining the approximation space this way is not unnatural since often in such methods, one considers interpolation using cardinal functions for different lattices.  Therefore, the use of shifted and dilated cardinal functions will correspond to interpolation of the wavelet at differently scaled lattices to form the approximant.
We note also that (A2) is satisfied for all cardinal functions since $|L_\alpha(x)|\leq1$ for all $x\in\R^d$ for the examples listed in the sequel.

\subsubsection{Gaussian Cardinal Function}

The cardinal function associated with the Gaussian kernel $\phi_\alpha$ described in section \ref{SECGaussian} was studied extensively by Riemenschneider and Sivakumar \cite{RS1,RS2,RS3,RS4,Siva}.  In particular, they showed that the cardinal function decays exponentially away from the origin, as does its Fourier transform, which implies (A5) with $\alpha_k''=0$ and (A6).  Additionally, for every $\alpha$, $|L_\alpha(x)|\leq1$ for every $x\in\R^d$.  Finally, (A3) follows from Lemma \ref{LEMCardinalA3}, while (A4) is a simple exercise based on the exponential decay of the Gaussian.

\subsubsection{Multiquadric Cardinal Functions}

Details on the behavior of the cardinal functions associated with the general multiquadrics may be found in \cite{HL1,HL2} for a broad range of exponents $\alpha$, whereas the particular cases of $\alpha=\pm1/2, 1, -k+1/2$, $k\in\N$ were considered previously \cite{Baxter,Buhmann,BM}.

In fact, checking the condition (A3) was done in Section 7 of \cite{HL1} for the univariate cardinal function, while a different estimate in \cite{HL2} demonstrates (A3) in higher dimensions. (A4) and (A5) also follow from those estimates (even with the same parameters $\alpha_k,\alpha'_k,\alpha''_k$ as in Section \ref{SECMQI}).  For any $\alpha>d/2$, the associated multiquadric cardinal function satisfies (A1)--(A6).

It should also be noted that the growing multiquadrics $\phi_{\beta,c}:=(|\cdot|^2+c^2)^\beta$ for $\beta\geq1/2$ have well-defined cardinal functions which also satisfy (A1)--(A6) \cite{HL1,HL2}.

\section{Regularity Criteria for Growing Kernels}\label{SECGrowing}

In this section we investigate the possibilities for growing kernels, which requires a reworking of conditions (A1)--(A6).  Moreover, the approximants defined in this section differ from those previously discussed, but nonetheless achieve the same approximation orders.  To begin, given $N\geq N_0$, we seek an $M:=M_N$ term approximation space
\[
\Phi_M:=\left\{\finsum{j}{1}{M}a_j\phi_{\alpha_j}(\cdot-x_j)\;:\;(a_j)\subset\C, (\alpha_j)\subset A, (x_j)\subset\R^d\right\}
\]
which is closed under translation and dilation.  Relying on the analysis above, we seek functions $\phi_{\alpha}$ such that a finite linear combination of their translates decay.  For the moment, consider the univariate case, where we could hope to use the divided difference to obtain (A2) and perhaps (A5).  In principle and practice, condition (A3) is the hardest to check.  Notice that a growing kernel leads to a singularity at the origin in the Fourier domain; however, the function $\sum_{j\in\mathbb{Z}}A_{\alpha,h,j}^{\phi}$ has no such singularity.  With this and the formulation in Proposition \ref{PROP1} in mind, we define the bivariate kernel $k_{\phi,h}$ whose Fourier transform in the first variable is given by
\begin{equation}\label{growing k def}
\widehat{k_{\phi,h}}(\xi,x):= \sum_{j\in\mathbb{Z}}\dfrac{\widehat\phi(\xi+2\pi j/h)e^{2\pi i j x/h}}{\widehat\phi(\xi) }\chi_B(\xi).
\end{equation}
By (A3), this kernel is well defined for all $x$, since
\[
\|  \widehat{k_{\phi,h}}(\cdot,x) \|_{L_{1}}\leq C\left(1+\sum_{j\neq 0}\| A_{h,j}^{\phi} \|_{L_{\infty}(B)}\right).
\]
Thus we make our approximation of $f\in \schwartz_B$ of the form
\begin{equation}\label{T sharp}
\widetilde{ T}^{\sharp}_{h}f(x):=[f*k_h(\cdot,x)](x)
\end{equation}
(i.e. the convolution is taken in the first variable) and note that when everything is smooth and decays well enough, the Poisson summation formula holds and we recover our previous work; that is, $\widetilde{T}_h^\sharp f=T_h^\sharp f$.  The advantage of this technique is that we can consider growing kernels from the outset.  For illustration purposes, consider the family of univariate multquadrics
\[
\tau_{\alpha}(x):=(x^2+c^2)^{\alpha-1/2}, \quad\alpha\in\mathbb{N}.
\]
We know that by taking sufficiently many divided differences, we will be able to obtain any polynomial rate of decay.  In fact,
$\phi_{\alpha}(x):=\lfloor \tau_\alpha \rfloor_{2\alpha}(x)=O(|x|^{-2\alpha-1})$, where $\lfloor f \rfloor_{2n}=\frac{1}{(2n)!}\sum_{j=-n}^{n} (-1)^{j+n}\binom{2n}{j+n}f(\cdot+j).$
Notice that for this choice of kernel, we have
\[
A_{\alpha,h,j}^{\phi}(\xi)=\dfrac{(1-\cos(\xi+2\pi j/h))^{\alpha} \widehat{\tau_{\alpha}}(\xi+2\pi j/h)   }{(1-\cos(\xi))^\alpha\widehat{\tau_\alpha}(\xi)},
\]
so that if we choose $h=1/N$, we get
\[
A_{\alpha,1/N,j}^{\phi}(\xi)=\dfrac{ \widehat{\tau_{\alpha}}(\xi+2\pi jN)   }{\widehat{\tau_\alpha}(\xi)}=A_{\alpha,1/N,j}^{\tau}(\xi).
\]
The interplay between the divided difference and the choice of $h=1/N$ is vital and allows us to proceed.  To do so, we need to truncate our approximation, so we set $k^{\flat}_{\phi,1/N}(\cdot, x)$ to be the function whose Fourier transform is given by
\[
  \widehat{k^{\flat}_{\phi,1/N}}(\xi,x):= \sum_{|j|\leq N^2}\dfrac{\widehat\phi(\xi+2\pi jN)e^{2\pi i xjN}}{\widehat\phi(\xi) }\chi_B(\xi);
\]
then we define
\begin{equation}\label{T flat}
\widetilde{ T}^{\flat}_{1/N}f(x):=[f*k^{\flat}_{\phi,1/N}(\cdot,x)](x).
\end{equation}

Henceforth we make the following assumptions:
\begin{enumerate}
    \item[(B1)] $\Phi_M$ is closed under translation and dilation. 
    \item[(B2)] $\phi_\alpha = \lfloor \tau_\alpha \rfloor_{n_\alpha} $ for some $(n_\alpha)\subset\N$; that is, the kernels are built out of divided differences.
    \item[(B3)] (A3) holds for each $\phi_\alpha$ and $h=1/N$.
    \item[(B4)] For every $k\in\N$ and $j\neq 0$, there exists $\tilde\alpha_k\in A$ and $C$ independent of $N$ such that for every $\alpha\geq\tilde\alpha_k$ and $0\leq l \leq k$,
                 $D^l A_{\alpha,1/N,j}^{\tau} \in L_{\infty}(B)$, and \[\sum_{j\neq 0} \| D^l A_{\alpha,1/N,j}^{\tau}  \|_{L_{\infty}(B)}\leq C.\]
\end{enumerate}

We now use these hypotheses to prove the results analogous to those found in Section \ref{SECCriteria}; note that here we have $B=[-R,R]$.
\begin{lemma}\label{lem_unbdded}
There exists $\alpha_0\in A$ such that for all $\alpha\geq\alpha_0$, $N>R/\pi$, and $f\in \schwartz_B$,
\[
|(1/N)\widetilde{T}_{1/N}^{\sharp}f(x)|\leq \|\widehat{f}\|_{L_1}\left(1+ g_\alpha(1/N)\right).
\]
\end{lemma}
\begin{proof}
Using \eqref{T sharp} and the inversion formula, we have
\[
|(1/N)\widetilde{T}_{1/N}^{\sharp}f(x)|\leq \left\| \widehat{f}  \widehat{k_{\phi,1/N}}(\cdot,x) \right\|_{L_1} \leq \|\widehat{f}\|_{L_1}\left(1+ g_\alpha(1/N)\right) ,  
\]
the last inequality coming from (B3).
\end{proof}


\begin{proposition}\label{PROP1_unbdded}
There exists $\alpha_0\in A$ such that for all $\alpha\geq\alpha_0$, $N>R/\pi$, and $f\in \schwartz_B$,
\[\|f-(1/N)\widetilde{ T}_{1/N}^{\sharp}f \|_{L_\infty}\leq C\|\widehat{f}\|_{L_1}g_\alpha(1/N),\]
where $C>0$ is independent of $f$ and $N$.
\end{proposition} 

\begin{proof}
We have, from (B3),
\[
\left|  f(x)-(1/N)\widetilde{ T}_{1/N}^{\sharp}f(x)   \right| \leq \left\|  \widehat{f} \sum_{j\neq 0}A_{\alpha,1/N,j}^{\phi}   \right\|_{L_{1}(B)} \leq C\|\widehat f\|_{L_{1}}g_\alpha(1/N).
\]
\end{proof}


\begin{lemma}\label{LEMUniformHBound_unbdded}
Let $k\in\N$ and $f\in \schwartz_B$.  There exists $\alpha_k\in A$, such that for sufficiently large $N$, there is a constant $C_{f,k}>0$ independent of $N$, so that for all $\alpha\geq\alpha_k,$
\[
\| f- \widetilde{ T}^{\flat}_{1/N}f  \|_{L_\infty}\leq C_{f,k} N^{-k}.
\]
\end{lemma}

\begin{proof}
On account of (B3) and Proposition \ref{PROP1_unbdded}, we have
\begin{align*}
\left|  f(x)-(1/N)\widetilde{ T}_{1/N}^{\flat}f(x)   \right| &\leq \left|  f(x)-(1/N)\widetilde{ T}_{1/N}^{\sharp}f(x)   \right| + \left| (1/N)\widetilde{ T}_{1/N}^{\sharp}f(x) -(1/N)\widetilde{ T}_{1/N}^{\flat}f(x)   \right|\\
&\leq C\| \widehat f \|_{L_1(B)}g_{\alpha}(1/N)+ (1/N)\|\widehat{f}\|_{L_1(B)} \sum_{|j|>N^2}\|A_{\alpha,1/N,j}^{\phi}   \|_{L_{\infty}(B)}\\
& \leq C_{f,k}N^{-k}.
\end{align*}

\end{proof}


\begin{proposition}\label{prop_ptwise_bnd_unbdded}
Let $k\in\N_0$ and $f\in \schwartz_B$.  There exists $\alpha^\prime_k\in A$ so that for all $\alpha\geq \alpha^\prime_k$ and sufficiently large $N$, there is a constant $C$, independent of $N$ such that
\[|f(x)- (1/N)\widetilde{ T}_{1/N}^{\flat}f(x) |\leq C N^{-k}(1+|x|)^{-k}.\]
\end{proposition}

\begin{proof}
Similar to Proposition \ref{prop_ptwise_bnd}, we need only check that if $|x|> 2/h$, then
\[
\left| f(x)- (1/N)\widetilde{ T}_{1/N}^{\flat}f(x)   \right| \leq C_{k,f}\left( 1+|x|   \right)^{-2k}.
\]
Note that since $f\in\mathscr{S}$, $|f(x)|\leq C(1+|x|)^{-2k}$. From (B4), when $\alpha\geq 2\tilde\alpha_k$, we have 
\begin{align*}
\left|(1+|x|)^{2k} (1/N) \widetilde{ T}_{1/N}^{\flat}f(x)   \right| &\leq C_{f,k} \sum_{|j|>N^2} \sum_{l=0}^{2k}\|D^{l}A_{\alpha,1/N,j}^{\phi}   \|_{L_\infty(B)}\\
&\leq  \tilde{C}_{f,k}.
\end{align*}
\end{proof}

For a given $k$, we set $\alpha_K :=\max\{ \alpha_0, \alpha^\prime_k, \alpha^\prime_k \}$.  Thus the conclusions of the results above hold for every $\alpha\geq \alpha_K$.
\begin{theorem}\label{THMGrowingNBound}
Let $\psi\in \schwartz_B$ be given.  Suppose $k\in\N_0$, $N \geq N_0$, and let $I$ be a cube.  Then there exists a constant $C>0$, independent of $N$ and $I$, such that for all $N$ large enough and suitably large $\alpha$
\[
\left|\psi_I(x) - \widetilde{T}^{\flat}_{1/N}\psi_I(x)    \right|\leq C N^{-k}\left( 1+\dfrac{|x-c(I)|}{\ell(I)} \right)^{-2k}.
\]
\end{theorem}
We need only double $\alpha_K$ to achieve the stated bound.

In order to obtain results in the multivariate case we use tensor products, which is sufficient due to the fact that the wavelets considered here are tensor products themselves.  The univariate estimates above lead to the main results of this section.  Note that $\Phi_M^{\otimes d}$ is the space of $d$--fold tensor products of functions in $\Phi_M$, i.e. $h\in\Phi_M^{\otimes 2}$ has the form $h(x,y)=f(x)g(y)$, for $f,g\in\Phi_M$.



\begin{theorem}\label{THMGrowingMain}
Suppose that $s>0$, $N \geq N_0 $, and $1\leq p < \infty$.  There is a constant $C_{p,s,d}>0$, $M_{s,N}\in\mathbb{N}$ so that for $f\in F_{\tau, q}^{s}$, with $1/\tau = 1/p+s/d$ and $1/q=1+s/d$, there is $S_f\in \Phi_{M}^{\otimes d}$ so that
\[
\| f-S_f  \|_{L_p}\leq C_{p,s,d} N^{-s/d}|f|_{F_{\tau, q}^{s}}.
\]
\end{theorem}

\begin{proof}
We begin with the wavelet expansion of $f$, given by 
\[
f=\sum f_I\psi_I,
\]
and define $S_f$ in terms of this expansion
\[
S_f:=\sum f_I \widetilde{T}^{\flat}_{1/N_I}\psi_I.
\]
Recall that $\widetilde{T}^\flat_{1/N_I}\psi_I= \sum_{j=1}^{N_I}a_{I,j}\phi\left((\cdot-c(I)/\ell(I))   \right) $, so choose $\phi=\phi_\alpha$ with $\alpha=\alpha_s$ large enough so that Theorem \ref{THMGrowingNBound} holds.  Notice also that we have used no more than $\sum M_{s,N_I}\leq M_{s,N}$ total centers.  If everything is smooth, we have $M_{s,N}=N$, while if we use divided differences, then we have $M_{s,N}=(2\alpha_s+1)^d(N^2+1)^d$ terms involving the original kernel.

Since the bound in Theorem \ref{THMGrowingNBound} justifies Lemma \ref{lem8_equiv} with no further restriction, we use this to obtain
\[
\| f-S_f  \|_{L_p} \leq C_{k,d}\left\| \sum_{I} \min\{1,c_I^{-k/d}\}|f_I|\chi_I  \right\|_{L_p}.
\]
This estimate also makes use of the tensor product structure of the wavelet and approximant $S_f$.  To complete the proof, we must have $k>s$, which provides the restriction on $\alpha_s$.  The rest of the proof follows the same reasoning given in the proof of Theorem \ref{THMLpMain}. 
\end{proof}

The conclusion of Theorem \ref{THMGrowingMain} does not completely resolve the problem of approximation from the space $\Phi_M$ as the construction therein relies on taking tensor products for approximation in dimension larger than 1.  It remains of interest to determine an approximation method from the space $\Phi_M$ in 1 dimension that does not require taking divided differences of the growing kernel and which can be extended in a natural way to higher dimensions.

It should be noted that the techniques of \cite{DR} give convergence rates as in Theorem \ref{THMGrowingMain} for certain growing kernels such as surface splines, but the proofs rely on the kernels inverting an elliptic differential operator of a certain type.  The approximants $\widetilde{T}_h^\sharp$ and $\widetilde{T}_h^\flat$ used here are defined in different ways than those in \cite{DR,HR} and allow more flexibility in the choice of kernel (indeed they need not be associated to any differential operator), and yet achieve a similar goal.

\section{Examples of Growing Kernels}\label{SECGrowingExamples}

\subsection{Multiquadrics}

We will show that the family of divided differences of multiquadrics mentioned above satisfies properties (B1)--(B4).  To wit, consider
\[
\phi_\alpha = \lfloor \tau_\alpha \rfloor_{2\alpha},
\]
where $\tau_\alpha(x)=(x^2+\alpha^2)^{\alpha-1/2}$, $\alpha\in\mathbb{N}$, $x\in\R$.

Property (B1) is satisfied by using the same reasoning as that for (A1) in Section \ref{SECMQI}.  Property (B2) is obvious, and note the fact that we want the divided difference of order $2\alpha$ of $\tau_\alpha$.  To see that (B3) holds, note that
\[
A_{\alpha,1/N,j}^{\phi}(\xi)=\dfrac{ \widehat{\tau_{\alpha}}(\xi+2\pi jN)   }{\widehat{\tau_\alpha}(\xi)}=A_{\alpha,1/N,j}^{\tau}(\xi),
\]
so using estimates from Lemma 1 in \cite{HL1}, we have
\[
\|A_{\alpha,1/N,j}^{\tau}(\xi)\|_{L_\infty(B)}\leq e^{2cR}e^{-2\pi c |j| N }(2|j|-1)^{-\alpha},
\]
hence (A3) is satisfied.
Estimates from \cite{HL2} show that $A_{\alpha,1/N,j}^{\tau}$ is ($2\alpha-1$)--times differentiable and Theorem 5.1 there may be adapted to our needs, showing that the sum in (B4) can be bounded independent of $h$.

\subsection{Power kernels}

Our next example is $\phi_\alpha=\lfloor\tau_\alpha\rfloor_{2\alpha}$, where $\tau_\alpha(x)=|x|^{\alpha}$ for $\alpha\in\mathbb{R}^{+}\setminus2\mathbb{N}$.  We note that 
\[
\phi_\alpha(x)=O\left( |x|^{-\alpha}\right),\quad |x|\to\infty.
\]
We find the Fourier transform in Section 8.3 of \cite{Wendland}:
\[
\widehat{\tau_\alpha}(\xi)=(2\pi)^{\frac{d}{2}}\dfrac{2^{\alpha+\frac{1}{2}}\Gamma(\frac{\alpha+1}{2})}{\Gamma\left(-\frac{\alpha}{2}\right)}|\xi|^{-\alpha-1},
\]
which allows us to move forward with our computations.  Conditions (B1) and (B2) are evident, so we begin with (B3):
\[
A_{\alpha,1/N,j}^{\phi}(\xi)=\dfrac{|\xi|^{\alpha+1}}{|\xi+2\pi j N|^{\alpha+1}},
\]
thus
\[
\| A_{\alpha,1/N,j}^{\phi}  \|_{L_\infty(B)}\leq \dfrac{R^{\alpha+1}}{N^{\alpha+1}\pi^{\alpha+1}} (2|j|-1)^{-\alpha-1},
\]
hence (B3) is satisfied.  To see (B4), we apply the quotient rule repeatedly and find that
\[
|D^l A_{\alpha,1/N,j}^{\phi}(\xi)|\leq C_{\alpha,l}N^{-l}(2|j|-1)^{-\alpha},
\]
hence as long as $\alpha>1$, (B4) is satisfied. 


\section{Sobolev Interpolation Using Cardinal Functions}\label{SECInterpolation}

Let us now look more closely at the example of using cardinal functions as the family of kernels.  Due to some existing theory, the use of cardinal functions allows us to obtain error estimates for the classical Sobolev smoothness spaces which are not included in the framework above due to the restrictions on the smoothness $s$ of the Triebel--Lizorkin space.  The method here is one of interpolation, whereby more details may be found in the cardinal interpolation literature mentioned previously. 

\subsection{Decaying Functions}

To begin, formally define an alternate approximant of $f$ via
$$I_h^\sharp f(x):=\sum_{j\in\Z^d}f(hj)L_{\tau(h)}\bigg(\frac{x}{h}-j\bigg),$$
where $L_{\tau(h)}$ is one of the cardinal functions discussed in Section \ref{SECCardinal}, and $\tau(h)=h^{2}$ in the case of the Gaussian, and $h^{-1}$ for the multiquadrics.  Note that this is different from $T_h^\sharp f$ due to the fact that we use the samples of $f$ at the lattice $h\Z^d$ in the approximant rather than the values of $f_{L_{\tau(h)}}$ as defined previously.  This object has been studied in various instances before \cite{Buhmann,HL2,HMNW}, and is actually an interpolant of $f$.  By definition of the cardinal functions, it is easy to see that $I_h^\sharp f(hk)=f(hk),\;k\in\Z^d$.  Note that for the general multiquadric cardinal functions, there is an additional parameter $\alpha$ governing the power of the multiquadric $(|x|^2+\tau(h)^2)^\alpha$.  We suppress this subscript, and simply consider $\alpha$ to be fixed but arbitrary in the range $(-\infty,-d-1/2)\cup[1/2,\infty)\setminus\N$.

It is known that for certain classes of Sobolev functions, these interpolants exhibit nice approximation rates.  For example, the following holds.

\begin{theorem}[\cite{HMNW}, Theorem 2.1 and \cite{HL2}, Theorem 3.1]\label{THMInterpRates}
Let $L_{\tau(h)}$ be the cardinal function associated with the Gaussian or the general multiquadrics. If $1<p<\infty$, and $k>d/p$, then there exists a constant $C$, independent of $h$, such that for every $f\in W_p^k(\R^d)$, 
$$\|I_h^\sharp f-f\|_{L_p}\leq Ch^k\|f\|_{W_p^k}.$$
If $p=1,\infty$, then the bound changes to $C(1+|\ln h|)^dh^k\|f\|_{W_p^k}$.
\end{theorem}

With theorem \ref{THMInterpRates} in mind, one would naturally desire some estimate for a related interpolant which makes use of only finitely many samples of the function.  While there are many feasible ways to do this, we focus here on a method similar to the preceding analysis; to wit, let $$I_h^\flat f(x):=\sum_{j\in\Z^d\cap B_h}f(hj)L_{\tau(h)}\left(\frac{x}{h}-j\right),$$
where $B_h$ is the ball of radius $h^{-2}$ centered about the origin.  Similarly, let $I_Nf$ be as in \eqref{EQTNdef} where $N\sim h^{-2d}$. 

For arbitrary Sobolev functions, it is difficult to ascertain the behavior of the truncated interpolant; however, if one assumes a certain asymptotic decay rate on the function itself, then something may be said.  Define $\schwartz_{\kappa}$ to be the class of functions on $\R^d$ with decay $O(|x|^{-\kappa})$, $|x|\to\infty$.  Then the following holds.

\begin{theorem}\label{THMSobolevNTerm}
With the notation and parameters as in Theorem \ref{THMInterpRates}, there exists a constant $C$, independent of $h$, such that for every $f\in W_p^k(\R^d)\cap\schwartz\kappa$,
$$\|I_h^\flat f-f\|_{L_p}\leq Ch^\rho\|f\|_{W_p^k},$$
where $\rho = min\{k,\kappa-2d+d/p\}$ for multiquadric interpolation, and $\rho=\min\{k,\kappa-d\}$ for Gaussian interpolation.
\end{theorem}

Before supplying the complete proof, we collect some useful lemmas.

\begin{lemma}\label{LEMSumBound}
If $f\in\schwartz_{\kappa}$, then $$\sum_{|j|>h^{-2}}|f(hj)|\leq Ch^{\kappa-2d}.$$
\end{lemma}

\begin{proof}
By the assumption on the decay of $f$ and the same estimate as in the proof of Lemma \ref{LEMUniformHBound}, we find that
$$\sum_{|j|>h^{-2}}|f(hj)|\leq C\sum_{|j|>h^{-2}}\frac{1}{|j|^\kappa h^\kappa}\leq C_dh^{-\kappa}\dint_{h^{-2}}^\infty r^{d-1-\kappa}dr = C_{d,\kappa}h^{-\kappa}h^{-2d+2\kappa},$$
which is at most $Ch^{\kappa-2d}$ as required.
\end{proof}

\begin{lemma}\label{LEMMultiplierBound}
Let $1<p<\infty$. If $L_{\frac{1}{h}}$ is the cardinal function associated with the general multiquadric, then 
$$\|L_{\frac{1}{h}}(\cdot/h)\|_{L_p}\leq Ch^{d/p},$$ whereas if $L_{h^2}$ is that associated with the Gaussian, then $$\|L_{h^2}(\cdot/h)\|_{L_p}\leq Ch^{d}.$$
\end{lemma}

\begin{proof}
Let us begin with the multiquadric case.  From \cite[Corollary 5.3]{HL2}, we see that $|L_{\frac{1}{h}}(x/h)|\leq C_{d}\min\{1,h^d|x|^{-d}\}$ (we also note that the estimate may be obtained in a straightforward manner from Section 4 of \cite{HL1}). Therefore, the $p$--th power of the $L_p$ norm in question is bounded by 
$$\dint_{B(0,h)}dx+\dint_{\R^d\setminus B(0,h)}\frac{h^{dp}}{|x|^{dp}}dx\leq Ch^d+Ch^{pd}\dint_h^\infty r^{-dp}r^{d-1}dr,$$
which is at most $C_{\alpha,p,d}h^d$.  

For the Gaussian case, one need only notice that the multivariate Gaussian cardinal function is nothing but the $d$--fold tensor product of the univariate version.  Consequently, we may use the bound of \cite[Eq. 4.4]{HMNW}, which says that for the univariate Gaussian cardinal function, $L_{h^2}(x/h)\leq C\min\{h,h|x|^{-1}\}$.  Consequently, $$\|L_{h^2}(\cdot/h)\|_{L_p(\R)}^p\leq Ch^p\dint_{-1}^1dx+Ch^p\dint_1^\infty\frac{1}{|x|^p}dx\leq Ch^p.$$
Thus the multivariate estimate is $Ch^{dp}$, whence taking $p$--th roots gives the desired inequality.

\end{proof}

With these ingredients in hand, we are now ready to supply the proof of the theorem.

\begin{proof}[Proof of Theorem \ref{THMSobolevNTerm}]
First note that $\|I_h^\flat f-f\|_{L_p}\leq \|I_h^\sharp f-f\|_{L_p}+\|I_h^\sharp f-I_h^\flat\|_{L_p},$ and the first term is majorized by $Ch^k$ by Theorem \ref{THMInterpRates}.  Now to estimate the second term, it follows from Minkowski's integral inequality \cite[Theorem 6.19, p. 194]{Folland} that
\begin{displaymath}
\begin{array}{lll}
\|I_h^\sharp f-I_h^\flat f\|_{L_p} & = & \left(\dint_{\R^d}\left|\dsum_{|j|>h^{-2}}f(hj)L_{\tau(h)}\left(\frac{x}{h}-j\right)\right|^pdx\right)^\frac{1}{p}\\
& \leq & \dsum_{|j|>h^{-2}}\left(\dint_{\R^d}|f(hj)|^p\left|L_{\tau(h)}\left(\frac{x}{h}-j\right)\right|^pdx\right)^\frac{1}{p}\\
& = & \dsum_{|j|>h^{-2}}|f(hj)|\|L_{\tau(h)}(\cdot/h)\|_{L_p}.\\

\end{array}
\end{displaymath}
From Lemmas \ref{LEMSumBound} and \ref{LEMMultiplierBound}, we see that for the multiquadric, $\|I_h^\sharp f-I_h^\flat f\|_{L_p}\leq Ch^{\kappa-2d+d/p}$, while the bound for the Gaussian is $Ch^{\kappa-d}$, thus completing the proof.
\end{proof}

\begin{corollary}\label{CORSobolevNTerm}
With the parameters as in Theorem \ref{THMSobolevNTerm}, if $N\sim h^{-2d}$, then for every $f\in W_p^k(\R^d)\cap\schwartz_\kappa$,
$$\|I_N f-f\|_{L_p}\leq CN^{-\rho/(2d)}\|f\|_{W_p^k}.$$
\end{corollary}

Evidently, the $N$--term interpolants considered here work well for Sobolev functions which decay away from the origin.  Of course, if a given function peaks far away from the origin but still decays away from the peak, one should translate the peak to the origin and then interpolate, in which case the same estimate as in Corollary \ref{CORSobolevNTerm} holds for $I_N(f_I)-f_I$.  

It follows easily from the lemmas above that one can still estimate the difference of the full and truncated interpolants in $L_1$ and $L_\infty$ and get analogous bounds in terms of $h$.  However, the method of proof of Theorem \ref{THMSobolevNTerm} is not sufficient to estimate approximation orders of $I_h^\flat f-f$ in these spaces due to the logarithmic term found in Theorem \ref{THMInterpRates}.  Moreover, it has proven elusive to estimate the truncated interpolant via other means.

\subsection{Greedy Interpolation}

Of interest, especially in light of the recent advances in Greedy approximations (e.g. \cite{Temlyakov}), is the {\em greedy interpolant} of $f$, which for a given lattice would be formed by keeping only the $N$ largest (in absolute value) samples $f(hj)$.  That is, $G_N f(x) = \sum_{j\in\Lambda}f(hj)L_{\tau(h)}(x/h-j)$, where $|f(hj)|\geq|f(hk)|$, $j\in\Lambda$, $k\notin\Lambda$, and $|\Lambda|=N$.  Inspired by the estimates in \cite{BD} for greedy quasi-interpolants, we may provide similar estimates here for the greedy Gaussian and multiquadric interpolants.   

The rest of the argument mimics that of \cite{BD}, although the proofs of the intermediate steps do not follow directly from their argument, but rather come from other known estimates for the Gaussian and multiquadric cardinal functions.  Before stating the theorem, we need an auxiliary lemma on the Lebesgue constants for the given cardinal functions.

\begin{lemma}\label{LEMLebesgueConstants}
Let $\alpha\in(-\infty,-d-1/2)\cup[1/2,\infty)\setminus\N$ be a fixed parameter for the general multiquadric, and let $L_{\tau(h)}$ be either the cardinal function associated to the multiquadric or the Gaussian.  Then for $h>0$, we have
\[\underset{y\in\R^d}\sup \;\sum_{j\in\Z^d}|L_{\tau(h)}(y-j)|\leq C|\ln h|^d, \]
for some $C$ independent of $h$.
\end{lemma}
\begin{proof}
The estimate is well-known for the Gaussian, and may be found in \cite[Theorem 5.2]{RS2}, while for the multiquadrics, the estimate in dimension 1 for general $\alpha$ is given by \cite[Proposition 7]{HL1} (though for $\alpha=1/2$, the estimate is in \cite{RS1}), and the proof of the bound in higher dimensions follows the same line of reasoning upon using the bounds provided in \cite[Corollary 4.7]{HL2}.
\end{proof}

\begin{theorem}\label{THMGreedy}
Let the parameters $p, k$ and $\alpha$ be as in Theorem \ref{THMInterpRates}, and let $G_N$ be the greedy Gaussian or multiquadric interpolant with parameter $\alpha$. If $h\sim N^{-\frac{1}{kp+d}}$, then
\[\|f-G_Nf\|_{L_\infty}\leq CN^{-\frac{k}{kp+d}}(1+|\ln N|)^d\|f\|_{W_p^k},\quad f\in W_p^k. \]
\end{theorem}
\begin{proof}
First of all, by Theorem \ref{THMInterpRates}, we have that \[\|f-I_h^\sharp f\|_{L_\infty}\leq CN^{-\frac{k}{kp+d}}(1+|\ln N|)^d\|f\|_{W_p^k},\]
whereby it suffices to estimate $\|I_h^\sharp f-G_Nf\|_{L_\infty}.$
The quantity in question is, by definition, \[\left\|\sum_{j\notin\Lambda} f(hj)L_{\tau(h)}\left(\frac{\cdot}{h}-j\right)\right\|_{L_\infty}.\]
Since $f\in W_p^k$, it follows from the Sobolev embedding Theorem that $\sum_{j\in\Z^d}|f(hj)|^p\leq Ch^{-d}\|f\|_{W_p^k},$ whereby (as noted in \cite{BD}, there can be at most $n$ elements of $\Z^d$ for which $|f(hj)|\geq Ch^{-\frac dp}N^{-\frac1p}\|f\|_{W_p^k}.$  Consequently, since $|\Lambda|=N$,  the reverse inequality is true for $|f(hj)|$ for all $j\notin\Lambda$.  Finally, we need only appeal to Lemma \ref{LEMLebesgueConstants} to see that 
\[\|I_h^\sharp f-G_Nf\|_{L_\infty}\leq C|\ln N|^dh^{-\frac{d}{p}}N^{-\frac1p}\|f\|_{W_p^k}. \]  

The result follows upon combining this estimate with that from Theorem \ref{THMInterpRates}, with the requirement on $h$.
\end{proof}

Of particular future interest would be to give estimates on $\|f-G_Nf\|_{L_p}$ in the above setting; however, this remains an elusive task at present.

\section{Cost Distribution}\label{SECCost}

\subsection{Examples}

In this section, we analyze what the cost distribution looks like for some particular types of functions.  Suppose that $\psi$ is the mother wavelet, and that $\psi_{I}$ is as defined previously for a dyadic cube $I$.  Then suppose that $\{I_j\}_{j=1}^M$ are disjointly supported dyadic cubes, and we consider the cost distribution for the function $$f=\sum_{j=1}^M a_j\psi_{I_j}.$$
We recall the relation of the parameters: for a fixed $p$ and $s$, we have $1/\tau = 1/p+s/d$, and $1/q=1+s/d$.  Then
$$M_{s,q}f(x) = \left(\finsum{j}{1}{M}|I_j|^{-\frac{sq}{d}}|a_j|^q\chi_{I_j}(x)\right)^\frac{1}{q}.$$
Therefore, since the $I_j$ are disjoint,
\begin{displaymath}
\begin{array}{lll}
|f|_{F_{\tau,q}^s} & = & \left(\dint_\R\left(\finsum{j}{1}{M}|I_j|^{-\frac{sq}{d}}|a_j|^q\chi_{I_j}(x)\right)^\frac{\tau}{q}dx\right)^\frac{1}{\tau}\\
& = & \left(\finsum{j}{1}{M}\dint_{I_j}|I_j|^{-\frac{s\tau}{d}}|a_j|^\tau dx\right)^\frac{1}{\tau}\\
& = & \left(\finsum{j}{1}{M}|a_j|^\tau|I_j|^\frac{\tau}{p}\right)^\frac{1}{\tau},\\
\end{array}
\end{displaymath}
the final step coming from the observation that $1-s\tau/d=\tau/p$.

Next, recall that $m_{s,q,I_j}:=M_{s,q,I_j} = |I_j|^{-\frac{s}{d}}|a_j|$, and by definition $f_{I_j}=a_j$.  Then the cost of the cube $I_j$ is
\begin{align}\label{EQCost}
    c_{I_j} &= \dfrac{1}{\finsum{m}{1}{M}|a_m|^\tau|I_m|^\frac{\tau}{p}}|I_j|^{-\frac{s}{d}(\tau-q)}|a_j|^{\tau-q}|a_j|^q|I_j|^q N \nonumber\\
    & = \dfrac{|a_j|^\tau|I_j|^\frac{\tau}{p}}{\finsum{m}{1}{M}|a_m|^\tau|I_m|^\frac{\tau}{p}}\;N.
\end{align}

Recalling that $N_{I_j}=\floor{c_{I_j}}$ if the right-hand side is at least $N_0$, and $N_{I_j}=0$ otherwise, we have that
\[S_f = \finsum{j}{1}{M}a_jT_{N_j}\psi_{I_j}.\]

As an example of the above, we illustrate the cost distribution for a function having a disjoint wavelet expansion.  Consider a fixed 7--term wavelet, $f = a_{1}\psi(2^{j_1}(\cdot+3))+a_2\psi(2^{j_2}(\cdot+2))+\dots+ a_7\psi(2^{j_7}(\cdot-3))$ where the coefficients $a_1,\dots a_7$ and the dilations were  randomly generated to yield coefficients $(9, 5, 5, 3, 8, 4, 1)$ and dilations of $(3, 1, 1, 2, 0, 0, 4)$. The number of terms is of course arbitrary, but the cost distribution is displayed in Table \ref{TABRandom} for increasing budget for the parameters $s=1, d=1, p=1, \tau=\frac12$.

\begin{table}[!h]
\begin{tabular}{| c | c | c |}
\hline $N$  & Cost $(c_1,\dots,c_7)$\\\hline
100 &   (10, 16, 16, 9, 28, 20, 2) \\\hline
200 &   (21, 31, 31, 17, 56, 39, 5)\\\hline
300 &   (31, 47, 47, 26, 83, 59, 7)\\\hline
400 &  (42, 62, 62, 34, 111, 79, 10)\\\hline
500 & (52, 78, 78, 43, 139, 98, 12)\\\hline
\end{tabular}

\caption{Cost distribution for $f$ with sparse wavelet representation.}
\label{TABRandom}
\end{table}

It can be seen from this example that dilations giving larger supports are more favored in the cost distribution, as the terms with dilation 1 and 0 have the largest associated cost, while those with large dilation (hence smaller measure of $|I_j|$) have smaller costs.  Additionally, the coefficients $a_j$ scale the cost according to the power of $\tau$, for example $c_5\approx \sqrt{2}c_6$, where $a_5=2a_6$ (recall that $\tau=\frac12$).

\subsection{Limitations}

Here, we present an inherent limit to this scheme.  It should not be surprising based on the discussion above, but the cost distribution scheme can fail to capture enough information on a signal whose wavelet expansion is very spread out.  Let us suppose that $N$ is fixed, and arbitrarily large.  We will exhibit a signal whose $N$--term approximant $S_{f,N}$ is identically 0 because the cost of each dyadic cube is 0.

Let $M>N^\tau$, and let $(I_j)_{j=1}^M$ be disjointly supported cubes of unit volume, and define $f=\finsum{j}{1}{M}\psi_{I_j}.$  By \eqref{EQCost}, we find that for each $1\leq j\leq M$, 
$$c_{I_j}=M^{-\frac{1}{\tau}}N<1,$$
whereby, $N_{I_j}=0$ for every $j$.

Finally, it should be noted that the theoretical cost distribution above is difficult to implement in practice.  In particular, it assumes knowledge of all the wavelet coefficients of the target function, $f$.  Moreover, computing the approximants is not typically fast in this case.  For the future, it would be interesting to consider other approximations from the spaces $\Phi_N$ above which may obtain the optimal recovery rates and which are more readily implemented numerically.  However, the main purpose of this work was to demonstrate criteria on general approximation spaces which yield best $N$--term approximation orders analogous to the wavelet decomposition of the space itself.

\begin{Acknowledgements*}
Work for this article was completed while the first author was an Assistant Professor at Vanderbilt University. In the latter stages, the first author was partially supported by the NSF TRIPODS program under grant CCF-1423411.  The second author would like to thank the Mathematics Department at Vanderbilt for hosting him during his stay, and supporting collaboration on this article.
\end{Acknowledgements*}



\begin{thebibliography}{00}


\bibitem{AandS}
M. Abramowitz and I. A. Stegun (Eds.), {\em Handbook of mathematical functions: with formulas, graphs, and mathematical tables.} No. 55, Courier Dover Publications, 1972.

\bibitem{AldroubiUnser} A. Aldroubi and M. Unser, Sampling procedures in function spaces and asymptotic equivalence with Shannon's sampling theory, {\em Numer. Funct. Anal. Optim.} \textbf{15}(1-2) (1994), 1-21.

\bibitem{Atreas} N. D. Atreas, On a class of non-uniform average sampling expansions and partial reconstruction in subspaces of $L_2(\R)$, {\em Adv. Comput. Math.} \textbf{36}(1) (2012), 21-38.

\bibitem{Baxter}
B. J. C. Baxter, The asymptotic cardinal function of the multiquadratic
$\phi(r)=(r^2+ c^2)^\frac{1}{2}$ as $c\to\infty$, {\em Comput. Math. Appl.} \textbf{24}(12) (1992), 1-6.

\bibitem{BS}
B. J. C. Baxter and N. Sivakumar, On shifted cardinal interpolation by Gaussians and multiquadrics, {\em J. Approx. Theory } {\bf 87}(1) (1996),  36-59.

\bibitem{DDR1} C. de Boor, R. DeVore, and A. Ron, Approximation from Shift-Invariant Subspaces of $L_2(\mathbb{R}^d)$, \emph{Trans. Amer. Math. Soc.} {\bf 341}(2) (1994), 787--806.     

\bibitem{DDR2} C. de Boor, R. DeVore, and A. Ron, The Structure of Finitely Generated Shift-Invariant Spaces in $L_2(\mathbb{R}^d)$, \emph{J. Funct. Anal.} {\bf 119} (1994), 37--78.  

\bibitem{DHR} C. De Boor, K. H\"{o}llig, and S. D. Riemenschneider, {\em Box splines}, Vol. 98. Springer Science \& Business Media, 2013.

\bibitem{Bownik} M. Bownik, The structure of shift-invariant subspaces of $L_2(\R^n)$, {\em J. Funct. Anal.} \textbf{177}(2) (2000), 282-309.

\bibitem{Buhmann}
M. D. Buhmann, Multivariate cardinal interpolation with radial-basis functions,
{\em Constr. Approx.} \textbf{6}(3) (1990), 225-255.

\bibitem{BuhmannBook} M. D. Buhmann, {\em Radial Basis Functions: Theory and Implementations}, Vol. 12. Cambridge University Press, 2003.

\bibitem{BM} M. D. Buhmann and C. A. Micchelli, Multiquadric interpolation improved, {\em Comput. Math. Appl.} \textbf{24}(12) (1992), 21-25.

\bibitem{BD}  M. D. Buhmann and F. Dai, Compression using quasi-interpolation, {\em Ja\'{e}n Journal on Approximation} \textbf{7}(2) (2015), 203-230.

\bibitem{Daubechies} I. Daubechies, {\em Ten lectures on wavelets}, Vol. 61. Philadelphia: Society for industrial and applied mathematics, 1992.

\bibitem{DV} L. de Carli and P. Vellucci, $p$--Riesz bases in quasi shift invariant spaces, arXiv: 1710.00702, To Appear in: \textit{Contemporary Mathematics:} Proceedings of the AMS Special Sessions'Frames, Harmonic Analysis and Operator Theor'  Eds: Y. Kim, S. K. Narayan, G.Picioroaga, and E. Weber.

\bibitem{DJP} R. DeVore, B. Jawerth, and V. Popov, Compression of wavelet decompositions, {\em Amer. J. Math.} \textbf{114} (1992), 737-785.

\bibitem{DR}
R. DeVore and A. Ron, Approximation using scattered shifts of a multivariate function, \emph{ Trans. Amer. Math. Soc.} {\bf 362}(12) (2010), 6205--6229.

\bibitem{Folland}
G. B. Folland, {\em Real Analysis: Modern Techniques and Their Applications}, Second Edition, John Wiley \& Sons, 1999.

\bibitem{G}
L. Grafakos, {\em Classical and Modern Fourier Analysis}, Pearson Education, Inc., Upper Saddle River, NJ, 2004.

\bibitem{GS} K. Gr\"{o}chenig and J. St\"{o}ckler, Gabor Frames and Totally Positive Functions, {\em Duke Math. J.} \textbf{162}(6) (2013), 1003-1031.

\bibitem{Hamm} K. Hamm, Approximation rates for interpolation of Sobolev functions via Gaussians and allied functions, {\em J. Approx. Theory} \textbf{189} (2015), 101-122.

\bibitem{HL1} 
K. Hamm and J. Ledford, Cardinal interpolation with general multiquadrics, {\em Adv. Comput. Math.} \textbf{42}(5) (2016), 1149-1186.

\bibitem{HL2}
K. Hamm and J. Ledford, Cardinal interpolation with general multiquadrics: convergence rates, {\em Adv. Comput. Math.} \textbf{44}(4) (2018), 1205--1233.

\bibitem{HLQuasi} K. Hamm and J. Ledford, On the structure and interpolation properties of quasi shift-invariant spaces, {\em J. Funct. Anal.} \textbf{274}(7) (2018), 1959--1992.

\bibitem{HMNW} T. Hangelbroek, W. Madych, F. Narcowich and J. D. Ward, Cardinal
interpolation with Gaussian kernels, {\em J. Fourier Anal. Appl.} \textbf{18}
(2012), 67-86.

\bibitem{HR}
T. Hangelbroek and A. Ron, Nonlinear approximation using Gaussian kernels, \emph{J. Funct. Anal.} {\bf 259} (2010), 203--219.


\bibitem{Jia} R. Q. Jia, Shift-invariant spaces on the real line, {\em Proc. Amer. Math. Soc.} \textbf{125}(3) (1997), 785-793.

\bibitem{Johnson} M. J. Johnson, On the approximation order of principal shift-invariant subspaces of $L_p(\R^d)$, {\em J. Approx. Theory} \textbf{91} (1997), 279-319.

\bibitem{Johnson2} M. J. Johnson, Scattered data interpolation from principal shift-invariant spaces, {\em J. Approx. Theory} \textbf{113}(2) (2001), 172-188. 

\bibitem{Kalton}
N. J. Kalton, S. Mayboroda, and M. Mitrea, Interpolation of Hardy-Sobolev-Besov-Triebel-Lizorkin spaces and applications to problems in partial differential equations, {\em Contemporary Mathematics}, \textbf{445} (2007), 121--178.

\bibitem{KP}
G. Kyriazis and P. Petrushev, New bases for Triebel-Lizorkin and Besov spaces, {\em Trans. Amer. Math. Soc.} \textbf{354} (2002), no. 2, 749--776.

\bibitem{Ledford}
 J. Ledford, Recovery of Paley-Wiener functions using scattered translates of regular interpolators, \emph{J. Approx. Theory} \textbf{173} (2013), 1–13.

\bibitem{LedfordCardinal}
J. Ledford, On the convergence of regular families of cardinal interpolators,
{\em  Adv. Comput. Math.} \textbf{41} (2015), 357--371.

\bibitem{LedfordSpline} J. Ledford, Convergence properties of spline-like cardinal interpolation operators acting on $\ell^p$ data, {\em J. Fourier Anal. Appl.} \textbf{23}(1) (2017), 229--244.

\bibitem{M}
Y. Meyer, {\em Wavelets and Operators}, Cambridge Stud. Adv. Math., Vol. 37. Cambridge University Press, 1992, translated from the 1990 French original by D.H. Salinger.

\bibitem{RS1} S. D. Riemenschneider and N. Sivakumar, On the
cardinal-interpolation operator associated with the one-dimensional
multiquadric, {\em East J. Approx.} \textbf{7}(4) (1999), 485--514.

\bibitem{RS2} S. D. Riemenschneider and N. Sivakumar, On cardinal interpolation by Gaussian radial-basis functions: properties of fundamental functions and estimates for Lebesgue constants, {\em J. Anal. Math.} \textbf{79} (1999), 33--61.

\bibitem{RS3} S. D. Riemenschneider and N. Sivakumar, Gaussian radial basis functions: Cardinal interpolation of $\ell_p$ and power growth data, {\em Adv. Comput. Math.} \textbf{11} (1999), 229--251.

\bibitem{RS4} S. D. Riemenschneider and N. Sivakumar, Cardinal interpolation by Gaussian functions: A survey, {\em J. Analysis} \textbf{8} (2000), 157--178.

\bibitem{Schoenberg} I. J. Schoenberg, {\em Cardinal Spline Interpolation},  Vol. 12. Society for Industrial and Applied Mathematics, Philadelphia, 1973.

\bibitem{Siva} N. Sivakumar, A note on the Gaussian cardinal-interpolation operator, {\em Proc. Edinb. Math. Soc. (2)} \textbf{40} (1997), 137-150.

\bibitem{Temlyakov} V. Temlyakov, {\em Greedy approximation}, Vol. 20. Cambridge University Press, 2011.

\bibitem{U1} M. Unser, Sampling--50 years after Shannon, {\em Proc. IEEE} \textbf{88}(4) (2000), 569-587.

\bibitem{U2} M. Unser and T. Blu, Cardinal exponential splines: part I--theory and filtering algorithms, {\em IEEE Trans. Sig. Process.} \textbf{53}(4) (2005), 1425-1438.

\bibitem{Wendland}
H. Wendland, {\em Scattered Data Approximation,} Vol. 17. Cambridge University Press, 2005.

\end{thebibliography}
\end{document}